\documentclass[11pt,letterpaper]{amsart}
\usepackage[parfill]{parskip}
\usepackage{amssymb}
\usepackage{amsmath}
\usepackage{amsthm}
\usepackage{amsfonts}
\usepackage{color}
\usepackage{graphicx, hyperref}
\usepackage{media9}
\usepackage[T1]{fontenc}

\usepackage{textcomp}
\usepackage{palatino}
\usepackage[text={6.5in, 8.5in}, centering]{geometry}

\newcommand{\abs}[1]{\left\lvert#1\right\rvert}

\newtheorem{theorem}{Theorem}
\newtheorem{lemma}[theorem]{Lemma}

\newtheorem{proposition}[theorem]{Proposition}
\theoremstyle{definition}
\newtheorem{definition}[theorem]{Definition}
\theoremstyle{remark}
\newtheorem{remark}[theorem]{Remark}

\begin{document}

\title{An algebraic reduction of the `scaling gap' in the Navier-Stokes
regularity problem}

\author{Zachary Bradshaw}
\address{Department of Mathematics, University of Arkansas}
\email{zb002@uark.edu}

\author{Aseel Farhat}
\address{Department of Mathematics, Florida State University}
\email{aseelfarhat842@gmail.com}

\author{Zoran Gruji\'c}
\address{Department of Mathematics, University of Virginia}
\email{zg7c@virginia.edu}

\date{\today}

\begin{abstract}
It is shown--within a mathematical framework based on the suitably defined scale of sparseness
of the super-level sets of the positive and negative parts of the vorticity components, and in the context
of a blow-up-type argument--that the 
ever-resisting `scaling gap', i.e., the scaling distance between a regularity criterion and a corresponding \emph{a priori} bound
(shortly, a measure of the super-criticality of the 3D Navier-Stokes regularity problem), can be reduced by an \emph{algebraic 
factor}; since (independent) fundamental works of Ladyzhenskaya, Prodi and Serrin as well as Kato and
Fujita in 1960s, all the reductions have been logarithmic in nature, regardless of the functional 
set up utilized. More precisely, it is shown that it is possible to obtain an \emph{a priori}
bound that is algebraically better than the energy-level bound, while keeping the corresponding regularity criterion at the
same level as all the classical regularity criteria. The mathematics presented was inspired by 
morphology of the regions of intense vorticity/velocity gradients observed in computational simulations 
of turbulent flows, as well as by the physics of turbulent cascades and turbulent dissipation.
\end{abstract}

\maketitle

\vspace{2in}

\section{Prologue}

The Navier-Stokes equations (NSE) describing the motion of three-dimensional (3D) incompressible,
viscous, Newtonian fluids read

\begin{equation}\label{up}
 \partial_t u + (u \cdot \nabla) u =  \nu \triangle u - \nabla p + f,
\end{equation}

supplemented with the incompressibility condition $\nabla \cdot u = 0$. Here, the vector field $u$ 
represents the velocity, the scalar field $p$ the pressure, a positive constant $\nu$ is the viscosity,
and the vector field $f$ an external force. 

\medskip

Henceforth, for simplicity of the exposition, $\nu$ will be normalized to 1, $f$ taken to 
be zero, and the spatial domain to be the whole space $\mathbb{R}^3$ (since
this is an irreversible system, the temporal variable lives in the interval $(0, \infty)$). In this
setting, the only relevant piece of data is the initial velocity configuration $u_0$.

\medskip

Mathematical theory of the 3D NSE emerged from the groundbreaking work of Leray in the 1930s
\cite{Leray1, Leray2, Leray3} in which he established global-in-time existence of 
weak solutions emanating from
the initial data with arbitrary large energy,
and charted the paths for the study of local-in-time and global-in-time existence and uniqueness of
strong solutions within the realms of arbitrary large and appropriately small initial data, respectively.

\medskip

A central problem in the theory ever since has been the possibility of singularity formation in the
system, and could be considered either from the point of view of possible singularity formation in a weak
solution or from the point of view of possible finite time blow-up of a local-in-time strong solution. 
This problem is usually referred to as the Navier-Stokes regularity problem.

\medskip

The Navier-Stokes regularity problem is supercritical, i.e., there is a `scaling gap'--with respect
to the unique scaling transformation realizing the scaling-invariance of the equations--between any known regularity criterion and
the corresponding \emph{a priori} bound. (Note that if $(u, p)$ is a solution to the 3D NSE on
$\mathbb{R}^3 \times (0, \infty)$, then for any $\lambda > 0$, the rescaled pair $(u_\lambda, p_\lambda)$
given by $u_\lambda(x,t) = \lambda u(\lambda x, \lambda^2 t)$ and
$p_\lambda(x,t) = \lambda^2 p(\lambda x, \lambda^2 t)$ is also a solution; in addition, this is the only such transformation.)

\medskip

If we focus on $L^\infty((0,T), X)$-type spaces,
perhaps the most illustrative example is in the $L^p$-scale where the regularity class is
$L^\infty((0,T), L^3)$ (\cite{ESS03}; {for generalizations see \cite{GKP,Phuc}}), and the
corresponding \emph{a priori} bound the Leray bound $L^\infty((0,T), L^2)$. Moreover, since 1960s, the scaling
gap has been of the fixed `size' in the sense that all the regularity classes are (at best) 
scaling-invariant while--in contrast--all the known \emph{a priori} bounds are
on the scaling level of the Leray bound.

\medskip

Recall that arguably the most classical (scaling invariant) regularity criteria are the Ladyzhenskaya-Prodi-Serrin
conditions given by
$u \in L^\frac{2g}{q-3}((0,T), L^q)$, for $3 < q \le \infty$, the other endpoint being the aforementioned
Escauriaza, Seregin, and Sverak class $L^\infty((0,T), L^3)$. In the vorticity realm, this designation goes to
the Beale, Kato, and Majda (BKM)-criterion (\cite{BKM84}),
$\omega \in L^1((0,T), L^\infty)$ where $\omega$ is the vorticity of the fluid given by 
$\omega = \, \mbox{curl} \, u$.

\medskip

One might try to collect additional clues about possible singularity formation in the 3D NSE
from the geometry of turbulent flows. One of the most prominent morphological signatures
of this geometry--revealed in high-resolution computational simulations of high
(and even moderate; see \cite{SchSree}) Reynolds number-flows-- is spatial intermittency of 
the regions of intense vorticity and velocity gradients.

\medskip

In order to better understand geometry of the flow mathematically, it is beneficial to
study the vorticity-velocity formulation of the 3D NSE,

\begin{equation}\label{ou}
 \partial_t \omega + (u \cdot \nabla) \omega =  \triangle \omega + (\omega \cdot \nabla) u.
\end{equation} 

The incompressibility implies 
that
$u$ can be reconstructed from $\omega$ by solving $\triangle u = - \ \mbox{curl} \ \omega$, leading
to the Biot-Savart law, and closing the system. 

\medskip

Let us start by noting that conceivably a first work in mathematical analysis of the 3D NSE that
was motivated by geometry of the flow was a work by Giga and Miyakawa \cite{GiMi} establishing
global-in-time existence of small, measure-valued solutions modeling the vortex filaments (the
space utilized is a scaling-invariant Morrey-Campanato space).

\medskip

A conceptually related program aimed at mathematical modeling and analysis of coherent structures
in fluid flows via the study of oscillating and localized solutions to the NSE was initiated by 
Meyer and his students (Brandolese, Cannone, Lemarie-Rieusset, Planchon, ...); this brought in
an array of scale-sensitive harmonic analysis tools such as wavelets, atomic decomposition and 
Littlewood-Paley analysis. A chapter in C.I.M.E. lecture notes by Meyer \cite{Mey}
provides an insight into many of the main ideas and results related to this program.

\medskip

The vortex-stretching term, $(\omega \cdot \nabla) u$, is responsible for a possibly unbounded 
growth of the vorticity magnitude in 3D; mechanically, the process of vortex stretching in conjunction
with incompressibility of the fluid and conservation of the angular momentum amplifies the
vorticity, formally, if we erase the vortex-stretching term in the equations, and the only nonlinear
term left is the advection term $(u \cdot \nabla) \omega$, an \emph{a
priori} bound on the $L^2$-norm of the vorticity--which suffices to bootstrap to infinite
smoothness--follows in one line ($(\omega \cdot \nabla) u$ is identically zero in 2D).

However, some caution is advised in labeling the vortex-stretching term as a clear-cut `bad guy'. While
the process of vortex-stretching does amplify the vorticity, it also produces locally anisotropic
small scales in the flow (diameter of a vortex tube, thickness of a vortex sheet, \emph{etc.}), and
this gives a chance to some form of a locally anisotropic diffusion to engage just on time to prevent
the further amplification of the vorticity magnitude and possible singularity formation. As a matter of fact, 
Taylor \cite{Tay37} concluded
his paper from the 1930s with the following thought on turbulent dissipation:

\medskip

``It seems that the stretching of vortex filaments must be regarded as the principal mechanical cause
of the high rate of dissipation which is associated with turbulent motion.''

\medskip

The stretching is also chiefly responsible for the phenomenon of local coherence of the vorticity
direction prominent in coherent vortex structures (e.g., in vortex tubes) which can be viewed as a local 
manifestation of the quasi low-dimensionality of the regions of intense fluid activity.
This observation led to the rigorous mathematical study
of locally anisotropic dissipation in the 3D NSE originating in the pioneering work by Constantin \cite{Co94} 
where he presented a singular integral representation 
of the stretching factor in the evolution of the vorticity magnitude, featuring a geometric
kernel depleted precisely by local coherence of the vorticity direction. This has been referred to
as `geometric depletion of the nonlinearity', and was a key in the proof of a fundamental result by
Constantin and Fefferman \cite{CoFe93}, stating that as long as the vorticity direction (in the regions of
intense vorticity) is Lipschitz-coherent, no finite time blow-up can occur.

The Lipschitz-coherence was later scaled down to $\frac{1}{2}$-H\"older \cite{daVeigaBe02}, followed up
by a complete spatiotemporal localization to an arbitrary small parabolic cylinder 
$B(x_0,2R) \times (t_0-(2R)^2, t_0)$
\cite{Gr09}, and 
the construction of scaling-invariant, local, hybrid geometric-analytic regularity classes in which 
the coherence of the vorticity direction is balanced against the growth of the vorticity magnitude
\cite{GrGu10-1}. In particular, denoting the vorticity direction by $\xi$, the following scaling-invariant 
regularity class is included,

\[
 \int_{t_0-(2R)^2}^{t_0} \int_{B(x_0,2R)} |\omega(x,t)|^2
 \rho_{\frac{1}{2}, 2R}^2 (x,t) \, dx \, dt < \infty
\]

where

\[
 \rho_{\gamma, r}(x,t) = \sup_{y \in B(x,r), y \neq x} \frac{\Bigl|\sin
 \angle \Bigl(\xi(x,t), \xi(y,t)\Bigr)\Bigr|}{|x-y|^\gamma};
\]

this is to be contrasted to an \emph{a priori} bound
obtained in \cite{Co90},

\[
 \int_0^T \int_{\mathbb{R}^3} |\omega(x,t)| |\nabla \xi (x,t)|^2 \,
 dx \, dt \le \frac{1}{2} \int_{\mathbb{R}^3} |u_0(x)|^2 \
 dx,
\]

which exhibits the same scaling as the Leray class $L^\infty((0,T), L^2)$;
another manifestation of the scaling gap.

\medskip

Let us return to the spatial intermittency of the regions of intense vorticity, more specifically,
to `sparseness' of the vorticity super-level sets.
This type of morphology motivated a recent work \cite{Gr13}
where it was shown that as long as the vorticity super-level sets, cut at a fraction of
the $L^\infty$-norm, exhibit a suitably defined property of sparseness at the scale
comparable to the radius of spatial analyticity measured in $L^\infty$,
no finite time blow-up can occur. (The proof was presented for the velocity formulation,
and the argument for the vorticity formulation is completely analogous.)

The concept of sparseness needed in the proof is essentially local existence of a
sparse direction (the direction is allowed to depend on the point) at the scale of
interest, and it will be referred to as `1D sparseness'. The precise definitions of 
1D sparseness and a related concept of `3D sparseness' can be found at the 
beginning of section 3, as well as a remark on their relationship. In the rest of this
section, sparseness will refer to 3D sparseness; we copy the definition 
below for convenience of the reader ($m^3$ denotes the 3-dimensional
Lebesgue measure).

Let $S$ be an open subset of $\mathbb{R}^3$, $\delta \in (0,1)$, $x_0$ a point in $\mathbb{R}^3$,
and $r \in (0, \infty)$. $S$ is {\it 3D $\delta$-sparse around $x_0$ at scale $r$} if 

\[
\frac{m^3\bigl(S\cap B(x_0,r)\bigr)}{m^3\bigl(B(x_0,r)\bigr)} \leq \delta.
\]

Since we are interested in the super-level sets of a vector field cut at a fraction of the
$L^\infty$-norm, at the scale comparable to some inverse power of the $L^\infty$-norm,
there are several parameters floating around: the fraction at which we cut 
(call it $\lambda$), the ratio of sparseness $\delta$ and a constant multiplying the inverse
power of the $L^\infty$-norm (call it $\frac{1}{c}$). More precisely, the class of functions
of interest, $X_\alpha$, is defined as follows.

\begin{definition}
For a positive exponent $\alpha$, and a selection of parameters $\lambda$ in $(0,1)$, $\delta$ in $(0,1)$
and $c_0>1$, the class of functions
$X_\alpha(\lambda, \delta; c_0)$ consists of bounded, continuous functions
$f : \mathbb{R}^3 \to \mathbb{R}^3$ subjected to the requirement that the super-level set

\begin{equation}\label{a1}
\biggl\{ x \in \mathbb{R}^3: \, |f(x)| > \lambda \|f\|_\infty\biggr\}
\end{equation}

is $\delta$-sparse around any point $x_0$ at scale  $\frac{1}{c} \frac{1}{\|f\|_\infty^\alpha}$,
for some $c, \frac{1}{c_0} \le c \le c_0$.
\end{definition}

It is instructive to make a brief comparison of the scale of classes $X_\alpha$ with the scale of
Lorentz $L^{p, \infty}$ spaces defined by the requirement that the distribution function of $f$,
the Lebesgue measure of the super-level set cut at level $\beta$, decays at least as 
the inverse $p$-power of $\beta$. There are two key differences:

(i) the definition of $X_\alpha$ involves a specific type of the super-level sets, cut at a fraction 
of the $L^\infty$-norm,
while the definition of $L^{p, \infty}$ involves all levels $\beta$,

(ii) the $L^{p, \infty}$ spaces are blind to any geometric properties
of the super-level sets, and measure only the rate of decay of the volume, while the classes $X_\alpha$
can detect the scale of sparseness of the super-level sets of the highest interest.

\begin{remark}
It is worth noting that the same phenomenon will occur in all rearrangement-invariant (with respect
to the Lebesgue measure) function classes, i.e., the classes that do not differentiate between the
functions having identical distribution functions. Typical examples are Lebesgue classes
$L^p$, general Lorentz classes $L^{p, q}$ and Orlicz classes $L^\phi$.
\end{remark}

Taking a closer look, on one hand, it is plain that

\begin{equation}\label{s_r}
 f \in L^{p, \infty} \implies f \in X_\alpha \ \ \mbox{for} \ \ \alpha=\frac{p}{3}
\end{equation}

(for a given selection of $\lambda$ and $\delta$, the tolerance parameter $c_0$ will depend
on the $L^{p, \infty}$-norm of $f$),
on the other hand, in the geometrically worst case scenario (from the point of view
of sparseness), the whole super-level set being clumped into a single ball, being in $X_\alpha$
is consistent with being in $L^{3 \alpha, \infty}$.

This provides a simple
`conversion rule' between the two scales, $\alpha = \frac{p}{3}$, determining the scaling 
signature of the classes $X_\alpha$.

Let us decode the scaling distance between the regularity criterion established
in \cite{Gr13} (in the vorticity formulation), and any corresponding \emph{a priori} bounds.

Recall that the scale of sparseness required is comparable to the lower bound 
on the radius of spatial analyticity evaluated at a suitable time $t$ prior to (possible) blow-up, namely,
$\frac{1}{c} \frac{1}{\|\omega(t)\|^\frac{1}{2}_\infty}$,
i.e., one needs $\omega$ in $X_\frac{1}{2}$.

In contrast, the best \emph{a priori} bound available is $\omega$ in $X_\frac{1}{3}$;
this follows simply from the \emph{a priori} bound on the $L^1$-norm of the vorticity (\cite{Co90}) and
Chebyshev's inequality (an $L \log L$-bound on the vorticity is available under a very weak 
assumption that the vorticity direction is in a local, logarithmically-weighted $BMO$ space that allows for
discontinuities \cite{BrGr13-2}).

Applying the scaling conversion rule, $X_\frac{1}{2}$ corresponds to $L^{\frac{3}{2}, \infty}$ and
$X_\frac{1}{3}$ to $L^{1, \infty}$; since these are precisely the endpoint classes within the 
$L^{p, \infty}$-scale in the vorticity formulation, we arrive at yet another manifestation of
the scaling gap (also, not surprising as the 
derivation of the \emph{a priori} bound, $\omega$ in $X_\frac{1}{3}$, uses no geometric 
tools).

\medskip

Back to the drawing board. Recall that computational
simulations of turbulent flows indicate a high degree of local anisotropy present in the vorticity field,
manifested through the formation of various coherent vortex structures (e.g., vortex 
sheets and vortex tubes).

A simple but key observation is that a vector field exhibiting a high
degree of local anisotropy is more likely to allow for a considerable discrepancy
in sparseness between the full vectorial super-level sets and the super-level sets of the
components. As a matter of fact, it is easy to construct locally anisotropic smooth vector fields for which
the former are not sparse at any scale, while the latter can feature any predetermined 
scale of sparseness.

This is our primary motivation for replacing the scale of classes $X_\alpha$ by the
scale of classes $Z_\alpha$ based on sparseness of the super-level
sets of the locally selected positive and negative parts of the vectorial components $f_i$,
$f_i^\pm$, for $i=1, 2, 3$.

\begin{definition}
For a positive exponent $\alpha$, and a selection of parameters $\lambda$ in $(0,1)$, $\delta$ in $(0,1)$
and $c_0>1$, the class of functions
$Z_\alpha(\lambda, \delta; c_0)$ consists of bounded, continuous functions
$f : \mathbb{R}^3 \to \mathbb{R}^3$ subjected to the following local condition. For $x_0$ in $\mathbb{R}^3$,
select the/a component $f_i^\pm$ such that $f_i^\pm(x_0) = |f(x_0)|$
(henceforth, the norm of a vector $v=(a, b, c)$, $|v|$, will be computed as $\max \{|a|, |b|, |c|\}$),
and require that the set

\[
 \biggl\{ x \in \mathbb{R}^3: \, f_i^\pm(x) > \lambda \|f\|_\infty\biggr\}
\]

is $\delta$-sparse around $x_0$ at scale $\frac{1}{c} \frac{1}{\|f\|_\infty^\alpha}$,
for some $c, \frac{1}{c_0} \le c \le c_0$. Enforce this for all $x_0$ in $\mathbb{R}^3$.
\end{definition}

\begin{remark}
Shortly, we require (local) sparseness of the/a locally maximal component only.
\end{remark}

(The scaling conversion rule remains the same, $\alpha = \frac{p}{3}$.)

{
\begin{table}
{
\begin{center}
\begin{tabular}{ | p{5.5cm} | p{5.5cm} | p{3.8cm} |}
\hline
 ~&~&~
\\
Regularity class
& \emph{A priori} bound  
& Energy class 
\\  ~&~&~ \\ \hline ~&~&
\\ $u\in L^\infty(0,T;L^3)$
& $u\in L^\infty(0,T;L^2)$ 
& $u\in L^\infty(0,T;L^2)$ 
\\  ~&~& \\ \hline ~&~&
\\ $\omega \in L^\infty(0,T;L^\frac 3 2)$ 
& $\omega\in L^\infty(0,T;L^1)$
& $\omega\in L^\infty(0,T;L^1)$ 
\\  ~&~& \\ \hline ~&~&
\\ $\omega(\tau) \in X_{\frac 1 2}$ for a particular $\tau<T$ 
&  $\omega(\tau) \in X_{\frac 1 3}$ a.e.~$\tau<T$
&  $\omega(\tau) \in X_{\frac 1 3}$ a.e.~$\tau<T$ 
\\  ~&~& \\ \hline ~&~&~
\\   $\omega(\tau) \in Z_{\frac 1 2}$ for a particular $\tau<T$ 
&  $\omega(\tau) \in Z_{\frac 2 5}$  for~$\tau<T$ whenever $\|\omega(\tau)\|_\infty$ is large enough
&  $\omega(\tau) \in Z_{\frac 1 3}$ a.e.~$\tau<T$ 
\\ ~  ~ & ~ & ~
\\\hline 
\end{tabular}
\bigskip \label{table}
\caption{The scaling distance between regularity and energy classes is consistent across the table. Generally, \emph{a priori} bounded quantities are in the energy class.  Remarkably, Theorem \ref{yay2} shows the \emph{a priori} bound in the $Z_\alpha$ class is \emph{algebraically better} than the energy class (the \emph{a priori} bound in $Z_\alpha$ is derived
in the context of a blow-up argument).}
\end{center}
}
\end{table}
}

Going back to the vorticity, since this is potentially a much weaker condition,
the two questions to examine are as follows.

(a) is it possible to obtain an analogous regularity criterion at the same scale 
of sparseness (compared to the full vectorial super-level sets), or at all?

(b) is it possible to establish an \emph{a priori} bound at a smaller scale 
of sparseness
(compared to the full vectorial case), in particular, in the sense of modifying the scaling
exponent $\alpha$ by an algebraic factor?

\medskip

Affirmative answers to (a) and (b) are presented in sections 3 (Theorem \ref{yay1}) 
and 4 (Theorem \ref{yay2}), respectively,
yielding a regularity criterion in $Z_\frac{1}{2}$ and an \emph{priori}
bound in $Z_\frac{2}{5}$ (with the identical selection of the parameters $\lambda$,
$\delta$ and $c_0$); since

\[
 \frac{1}{3} < \frac{2}{5} < \frac{1}{2},
\]

this represents an algebraic reduction of the scaling gap in the $Z_\alpha$-framework.

\medskip

For a quick scaling comparison, several pointwise-in-time classes and their respective bounds are 
listed in Table \ref{table}.

\begin{remark}
Note that in the geometrically worst case scenario (for sparseness), $\omega$
being in $Z_\frac{2}{5}$ is consistent with $\omega$ being in $L^{\frac{6}{5}, \infty}$ which
is--in turn--at the same scaling level as $L^\frac{6}{5}$. Hence, a `geometrically blind' 
scaling counterpart of our \emph{a priori} bound in $Z_\frac{2}{5}$ would be the 
bound $\omega$ in $L^\infty((0,T), L^\frac{6}{5})$ which is well beyond 
state-of-the-art ($\omega$ in $L^\infty((0,T), L^1)$ \cite{Co90}). 
A counterpart in the velocity realm would be the bound $u$ in $L^\infty((0,T), L^\frac{12}{5})$,
also well beyond state-of-the-art (this would represent an algebraic improvement of the Leray bound
$u$ in $L^\infty((0,T), L^2)$).

\end{remark}

\begin{remark}
One should remark that this is not the first time that the exponent $\frac{2}{5}$ appeared
in the study of formation of small scales in turbulent flows. In the context of the
3D inviscid dynamics, the BKM-type criteria reveal
that the formation of infinite spatial gradients/infinitesimally-small spatial scales is necessary for a
finite time blow-up. In particular (see a discussion in \cite{Co94}), if we consider the case of approximately
self-similar blow-up for the 3D Euler in the form
\[
 \omega(x,t) \approx \frac{1}{T-t} \Omega\biggl(\frac{x}{L(t)}\biggr),
\]
then the relevant BKM-type necessary condition for the blow-up reads
\[
 \int_0^T \bigl(L(t)\bigr)^{-\frac{5}{2}} \, dt = \infty,
\]
and if we focus on the algebraic-type dependence,
\[
 L(t)=L(0) \Bigl(1-\frac{t}{T}\Bigr)^p,
\]
then the blow-up can be sustained only when $p \ge \frac{2}{5}$. 
Furthermore, since by dimensional analysis
$\omega \approx \frac{1}{T-t}$, this scaling is fully consistent with
the vorticity being in $Z_\frac{2}{5}$ near a possible singular time
$T$, strongly hinting at the inviscid nature of $Z_\frac{2}{5}$.
\end{remark}

\begin{remark}
It is not inconceivable that--at least partially--the discrepancy between the exponents 
$\frac{2}{5}$ and $\frac{1}{3}$
can be viewed as an indirect measure of the degree of local anisotropy of the vorticity field
in the regions of intense fluid activity. It is worth noting that the same line of arguments,
applied to the velocity itself, does not yield any reduction of the scaling gap (\cite{FGL});
this is consistent with the picture painted by computational simulations (and fluid
mechanics) in which the regions of high velocity magnitude tend to be more isotropic
(recall that in a suitable statistical sense, the K41 turbulence phenomenology--based
on the velocity--is isotropic and homogeneous, while one way to rationalize the phenomenon
of intermittency in turbulence is precisely via the coherent vortex
structures).
\end{remark}

\bigskip

\section{Local-in-time spatial analyticity of the vorticity in $L^\infty$}

\medskip

One of the most compelling manifestations of diffusion in the 3D NSE is the instantaneous
spatially-analytic smoothing of the rough initial data. An explicit, sharp lower bound on the radius
of (spatial) analyticity of the solution emanating from an initial datum in $L^p$, for 
$3 < p < \infty$, was given in \cite{GrKu98}; the method--inspired by the so-called Gevrey class-method
introduced by Foias and Temam in \cite{FT}--was based on complexifying the
equations and tracking the boundary of the (locally-in-time expanding) domain of analyticity 
via solving a real system of PDEs satisfied by the real and the imaginary parts of the 
complexified solutions. 

Since the Riesz transforms are not bounded on $L^\infty$, to obtain an estimate in the case $p=\infty$ 
without a logarithmic correction requires
a different argument (on the real level), see, e.g., \cite{GIM, Ku2} and \cite{Gu}
in the real and the complex setting, respectively, in the realm of the velocity-pressure
formulation of the 3D NSE.

The $L^\infty$-argument within the vorticity-velocity description requires a modification, and we present 
a sketch here, including an estimate on the vortex-stretching term, for
completeness. What follows is a modification of the exposition given in \cite{Ku2}.
In addition to the initial vorticity $\omega_0$ being bounded, a suitable decay of  $\omega_0$ at infinity will 
be required (we chose $\omega_0$ in $L^2$ for
convenience); however, it is worth noting that this is a `soft assumption', i.e., there will be no quantitative 
dependence on $\|\omega_0\|_2$ in the proof.

\medskip

\begin{theorem}[real setting]

Let the initial datum $\omega_0$ be in $L^2 \cap L^\infty$.
Then there exists a unique mild solution $\omega$ in $C_w\bigl([0,T], L^\infty \bigr)$ where
$T \ge \frac{1}{c} \frac{1}{\|\omega_0\|_\infty}$ for an absolute constant $c > 0$.

\end{theorem}

\begin{remark}
Since $L^\infty$ is not separable, the continuity of the heat semigroup at the initial time is
guaranteed only in the weak sense; hence the appearance of the space
$C_w\bigl([0,T], L^\infty \bigr)$ in the statement of the theorem. Alternatively, one could
replace $L^\infty$ with its (closed) subspace of the uniformly continuous functions, and
have the strong continuity at the initial time.
\end{remark}

\emph{Sketch of the proof.} \ The vorticity-velocity formulation of the 3D NSE in the components
is as follows,

\begin{equation}\label{vort}
\partial_t \omega_j - \triangle \omega_j + u_i \partial_i \omega_j = \omega_i \partial_i u_j, \ j=1,2,3
\end{equation}

where $u$ can be recovered from $\omega$ by the Biot-Savart law,

\[
 u_j(x,t) = c \int \epsilon_{j,k,l} \, \partial_{y_k}  \frac{1}{|x-y|} \, \omega_l(y,t) \, dy
\]

($\epsilon_{j,k,l}$ is the Levi-Civita symbol), making this a closed system
for $\omega$. Differentiation yields

\begin{equation}\label{r1}
\partial_i u_j (x,t) = c \, P.V. \int \epsilon_{j,k,l} \frac{\partial^2}{\partial_{x_i} \partial_{y_k}} 
\frac{1}{|x-y|} \omega_l (y,t) \, dy.
\end{equation}

A key observation is that the kernel

\[
  \frac{\partial^2}{\partial_{x_i} \partial_{y_k}} \frac{1}{|x-y|}
\]

is a classical Calderon-Zygmund kernel; hence, 

\begin{equation}\label{r2}
 \|\nabla u (\cdot, t)\|_{BMO} \le c \, \|\omega (\cdot, t)\|_{BMO}.
\end{equation}

Since we are interested in the mild solutions, we rewrite the equations (\ref{vort}) via the action
of the heat semigroup,

\begin{align}\label{mild}
 \omega_j (x,t) &= \int G(x-y, t) (\omega_0)_j (y) \, dy \notag\\
                        &- \int_0^t \int G(x-y, t-s) \, u_i \partial_i \omega_j (y,s) \, dy \, ds \notag \\
                        &+ \int_0^t \int G(x-y, t-s) \, \omega_i \partial_j u_j (y,s) \, dy \, ds,
\end{align}
                        
where $G$ is the heat kernel.

\medskip

Similarly as in the velocity-pressure description (see, e.g., \cite{GrKu98}), the estimates are performed on the sequence of smooth (entire in the spatial variable) global-in-time approximations.

\begin{align}\label{mild-n}
 \omega^{(n)}_j (x,t) &= \int G(x-y, t) (\omega_0)_j (y) \, dy \notag\\
                        &- \int_0^t \int G(x-y, t-s) \, u^{(n-1)}_i \partial_i \omega^{(n-1)}_j (y,s) \, dy \, ds \notag\\
                        &+ \int_0^t \int G(x-y, t-s) \, \omega^{(n-1)}_i \partial_j u^{(n-1)}_j (y,s) \, dy \, ds, 
\end{align}

supplemented with

\[
 u^{(n)}_j(x,t) = c \int \epsilon_{j,k,l} \, \partial_{y_k}  \frac{1}{|x-y|} \, \omega^{(n)}_l(y,t) \, dy.
\]

\medskip

The goal is to keep $\omega^{(n)}$ in $L^\infty$ and $\nabla u^{(n)}$ in $BMO$
(uniformly in $n$, locally-in-time), analogously
to keeping $u^{(n)}$ in $L^\infty$ and $p^{(n)}$ in $BMO$ in the velocity-pressure formulation.

We present an estimate on the vortex-stretching term $ \omega^{(n-1)}_i \partial_j u^{(n-1)}_j$, the term responsible for
a potentially unrestricted growth of the vorticity magnitude, in some detail.

\medskip

\texttt{An estimate on the vortex-stretching term.}

\medskip

The main ingredients are: a pointwise estimate on the heat kernel $G$,
$
 G(x,t) \le c \, \frac{\sqrt{t}}{\bigl(|x|+\sqrt{t}\bigr)^4}$ (see, e.g., \cite{LR}),
a property of a scalar-valued $BMO$ function $f$ featuring a suitable decay at infinity
(e.g., being in the closure of the test functions in the uniformly-local $L^p$ for some $p$, $1 \le p < \infty$),
$
 \int \frac{|f(x)|}{(|x|+1)^4} \, dx \le c  \, \|f\|_{BMO}$
(in general, one has to subtract a local average, for example, over a unit ball $B$, in which case the inequality takes the form
$
\int \frac{|f(x) - \frac{1}{|B|} \int_B f|}{(|x|+1)^4} \, dx \le c \, \|f\|_{BMO}$
\cite{St}), and 
the Calderon-Zygmund relation (\ref{r1}).

\medskip

Combining the above ingredients implies the following string of inequalities.

\begin{align}
\biggl| \int_0^t \int G(x-y, & t-s) \, \omega^{(n)}_i(y,s) \,  \partial_i u^{(n)}_j(y,s) \, dy \, ds\biggr|\notag\\
   & \le \sup_{s \in (0,t)} \|\omega^{(n)}(s)\|_\infty \, \int_0^t \int \frac{\sqrt{t-s}}{\bigl(|x-y| + \sqrt{t-s}\bigr)^4} \, 
   |\partial_i u^{(n)}_j| \, dy \, ds \notag\\
   &  \le \sup_{s \in (0,t)} \|\omega^{(n)}(s)\|_\infty \, \int_0^t \int \frac{1}{\bigl(|z|+1\bigr)^4} \, 
   |\partial_i u^{(n)}_j| \, dz \, ds \notag\\
   & \le \, c \, t \,  \sup_{s \in (0,t)} \|\omega^{(n)}(s)\|_\infty  \sup_{s \in (0,t)} \|\partial_i u^{(n)}_j(s)\|_{BMO}\notag\\
   & \le \, c \, t \,  \sup_{s \in (0,t)} \|\omega^{(n)}(s)\|_\infty  \sup_{s \in (0,t)} \|\omega^{(n)}(s)\|_{BMO} \notag\\
   & \le \, c \, t \,  \Bigl( \sup_{s \in (0,t)} \|\omega^{(n)}(s)\|_\infty \Bigr)^2.
\end{align}
  
\medskip

(Since $\omega^{(n)}(s)$ is in $L^2$, $\partial_i u^{(n)}_j(s)$ is 
in $L^2$ as well (by Calderon-Zygmund), and the first $BMO$-bound is available.)

\medskip

The estimate on the advection term can be absorbed in the above bound, and we arrive at 

\[
  \sup_{s \in (0,t)} \|\omega^{(n+1)}(s)\|_\infty \le c \, \|\omega_0\|_\infty + c \,  t \, \Bigl(\sup_{s \in (0,t)} \|\omega^{(n)}(s)\|_\infty\Bigr)^2, 
\]

which yields boundedness of the sequence on the time-interval of the desired length. A similar argument implies
the contraction property in $C_w\bigl([0,T], L^\infty \bigr)$, concluding the proof.

\medskip

The complexification argument is analogous to the complexification of the velocity-pressure 
formulation; for details, see \cite{GrKu98}.

The main idea is as follows. Recall that the functions in the sequence
of approximations $(\omega^{(n)}, u^{(n)})$ 
are entire functions in the spatial variable for any $t>0$. Then, denoting a point in
$\mathbb{C}^3$ by $x + i y$,
introduce a related sequence $(\Omega^{(n)}, U^{(n)})$ by

\[
\Omega^{(n)}(x,t) = \omega^{(n)}(x + i \alpha t, t)   \ \ \ \mbox{and} \ \ \ \ U^{(n)}(x,t) = u^{(n)}(x + i \alpha t, t)
\]

where $\alpha > 0$ is an optimization parameter. 

\medskip

The real and the imaginary parts of each of the elements in the sequence $(\Omega^{(n)}, U^{(n)})$ 
satisfy a real system in which the principal
linear parts are decoupled, and the only new type of terms generated by the complexification
procedure are two lower-order (first order) linear terms coming from the chain rule and the
Cauchy-Riemann equations in $\mathbb{C}^3$.

\medskip

Local-in-time estimates on the system are then of the same type as the estimates on the real 3D NSE, 
resulting in expansion of the domain of analyticity--after the optimization in the parameter
$\alpha$--comparable to $\sqrt{t}$ (the only global-in-time property needed is for the real and imaginary
parts of each $\Omega^{(n)}$, on each real-space `slice', to be in $L^2$; this is, as in the
real case, guaranteed by the assumption that $\omega_0$ in $L^2$).

More precisely, we arrive at the following result.

\medskip

\begin{theorem} [complex setting]\label{cpx}

Let the initial datum $\omega_0$ be in $L^2 \cap L^\infty$,
and $M$ a constant 
larger than 1.
Then there is a constant $c(M)>1$  such that there exists a unique mild solution $\omega$ in $C_w\bigl([0,T], L^\infty \bigr)$ where
$T \ge \frac{1}{c(M)} \frac{1}{\|\omega_0\|_\infty}$, and for any $t$ in $(0, T]$ the solution $\omega$ is
the $\mathbb{R}^3$-restriction of a holomorphic function $\omega$ defined in the domain

\[
 \Omega_t = \bigg\{ x+iy \in \mathbb{C}^3: \, |y| < \frac{1}{\sqrt{c(M)}} \, \sqrt{t} \bigg\};
\]

moreover, $\|\omega(t)\|_{L^\infty(\Omega_t)} \le M \, \|\omega_0\|_\infty$.

\end{theorem}

\begin{remark}\label{loc}
The above theorem can be localized in the spirit of \cite{BGK15}; this would provide a (local) lower bound on the radius 
of spatial analyticity in terms of the local quantities only (at the expense of a logarithmic correction).
\end{remark}

\bigskip

\section{Spatial complexity of the vorticity components:
a regularity criterion}

\medskip

Henceforth, $m^n$ will denote the n-dimensional Lebesgue measure. Let us start with
recalling the concepts of 1D and 3D sparseness suitable for our purposes (\cite{Gr13, FGL}).

\medskip

Let $S$ be an open subset of $\mathbb{R}^3$, $\delta \in (0,1)$, $x_0$ a point in $\mathbb{R}^3$,
and $r \in (0, \infty)$.

\medskip

\begin{definition}\label{1D_sparse} 
$S$ is {\it 1D $\delta$-sparse around $x_0$ at scale $r$} if there exists a unit vector 
${\bf d}$ in $S^2$ such that 
\[
\frac{m^1\bigl(S\cap (x_0-r{\bf d}, x_0+r{\bf d})\bigr)}{2r} \leq \delta.
\]
\end{definition} 

\medskip

\begin{definition}\label{3D_sparse} 
$S$ is {\it 3D $\delta$-sparse
around $x_0$ at scale $r$} if 
\[
\frac{m^3\bigl(S\cap B(x_0,r)\bigr)}{m^3\bigl(B(x_0,r)\bigr)} \leq \delta.
\]
\end{definition}

\medskip

\begin{remark}
On one hand, it is straightforward to check that, for any $S$, 3D $\delta$-sparseness at scale $r$
implies 1D $(\delta)^\frac{1}{3}$-sparseness at scale $\rho$, for some $\rho$ in $(0, r]$,
at any given pair $(x_0, r)$. On the other hand, there are plenty of simple examples in which any attempt
at the converse fails.
However, if one is interested in sparseness of a set at scale $r$, uniformly in $x_0$ (as we will be), the difference is 
much less pronounced.
\end{remark}

\medskip

We will formulate our regularity criterion as a no-blow-up criterion for a solution $\omega$ 
belonging to $C\bigl( [0,T^*), L^\infty\bigr)$ where $T^*$ is the first 
(possible) blow-up time. In this setting, it is convenient
to have the notion of an `escape time'.

\medskip

\begin{definition}\label{escape_time} 
Let $\omega$ be in $C\bigl( [0,T^*), L^\infty\bigr)$ where $T^*$ is the first possible blow-up time. 
A time $t$ in $(0, T^*)$ is an
\emph{escape time} if $\|\omega(\tau)\|_\infty > \|\omega(t)\|_\infty$ for any $\tau$
in $(t, T^*)$.
\end{definition}

\begin{remark}
It is a well-known fact that for any level $L > 0$, there exists a unique escape time $t(L)$;
this follows from the local-in-time well-posedness in a 
straightforward manner. As a consequence, there are continuum many escape times in $(0, T^*)$.
\end{remark}

\medskip

As in the previous section, we 
assume that the initial datum $\omega_0$ is (in addition) in $L^2$.

Then we can solve the 3D NSE locally-in-time, at any $t$ in $(0, T^*)$, and uniqueness in conjunction
with Theorem \ref{cpx} will yield a lower bound on the radius of spatial analyticity near $t$. In particular,
for an escape time $t$, we consider a temporal point $s=s(t)$ conforming to the 
requirement

\begin{equation}\label{s}
s \ \ \mbox{in} \ \ \biggl[t+\frac{1}{4c(M) \|\omega(t)\|_\infty},
t+\frac{1}{c(M) \|\omega(t)\|_\infty}\biggr];
\end{equation}

then, a lower bound on the radius of spatial analyticity at $s$ is given by

\[
\frac{1}{2 c(M) \|\omega(t)\|_\infty^\frac{1}{2}}.
\]

Moreover, since $t$ is an escape time, this lower bound can be replaced by

\[
\frac{1}{2 c(M) \|\omega(s)\|_\infty^\frac{1}{2}},
\]

i.e., in this situation there is no need to worry about a time-lag--a lower
bound on the analyticity radius at $s$ is given in terms of the $L^\infty$-norm
of the solution at $s$.

\medskip

The property of 1D sparseness will be assumed on the super-level sets of the positive and negative parts
of the vorticity components $\omega_j$ at $s$ (recall that for a scalar-valued function $f$, the positive and
negative parts of $f$ are given by $f^+(x)=\max\,(f(x),0)$ and $f^-(x)=-\min\,(f(x),0)$, respectively); 
more precisely, the sets of
interests will be the sets $V^{j, \pm}_s$ defined by

\begin{equation}\label{V}
 V^{j, \pm}_s=\biggl\{x \in \mathbb{R}^3: \, \omega_j^\pm(x,s) > \frac{1}{2M}\|\omega(s)\|_\infty\biggr\} 
\end{equation}

($M$ is as in Theorem \ref{cpx}).

\medskip

\begin{center} 
\begin{figure}
\includegraphics[scale=.36]{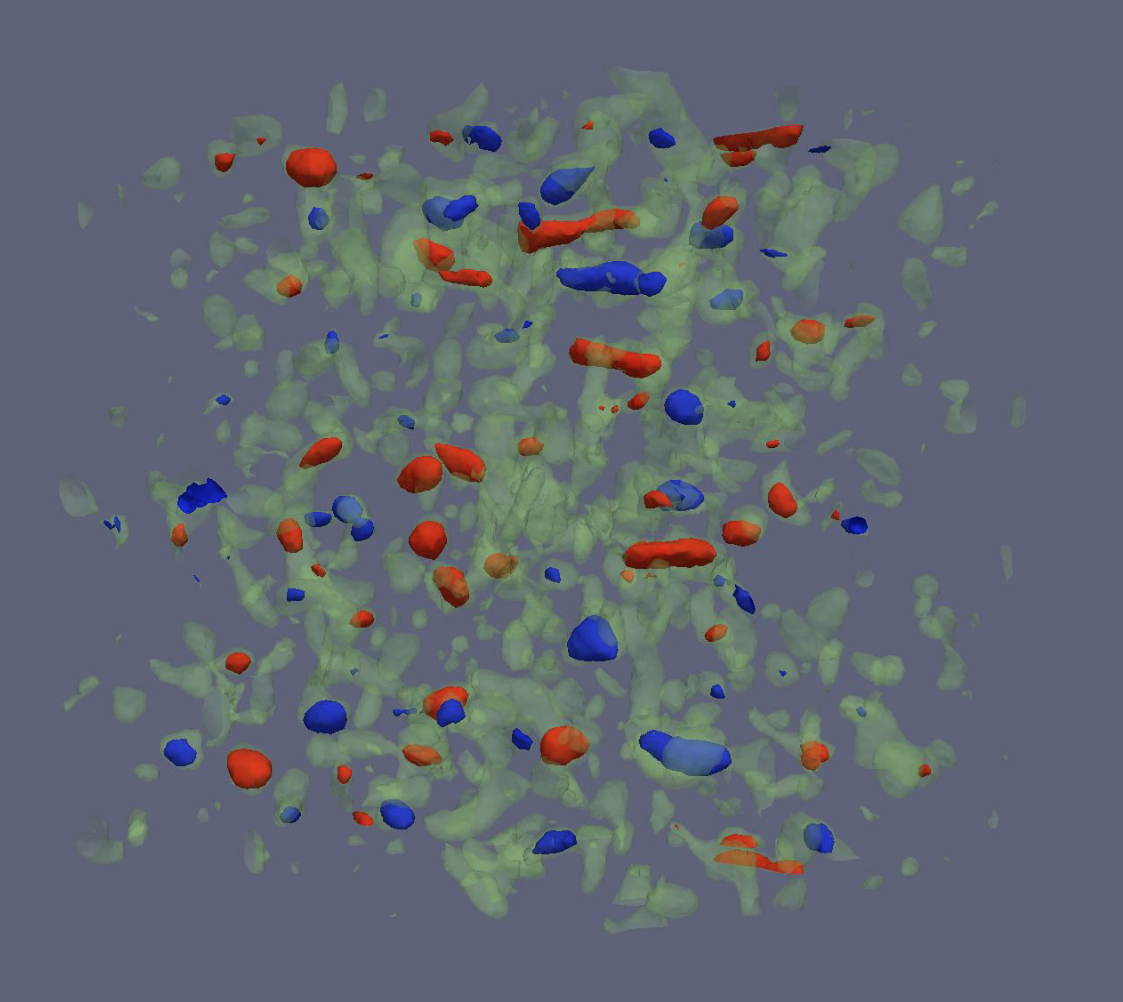}
\caption{\small{The blue and the red regions are precisely the sets $V^{1, \pm}$, i.e., the super-level sets of the positive and 
the negative parts of $\omega_1$, to be contrasted to the totality of the region visualized which corresponds to the super-level
set of the the full vectorial norm $\max \{|\omega_1|, |\omega_2|, |\omega_3|\}$ (from a
  snapshot of a simulation 
  initialized with the low frequency noise-type initial data at the Reynolds number of approximately $10^4$; courtesy of
  M. Mizstal, NBI).}}
 \end{figure}
 \end{center}

\medskip

Since the proof of our regularity criterion is based on the harmonic measure maximum principle,
we recall a version suitable for our purposes. Here, as in \cite{Gr13}, we are concerned with the harmonic
measure in the complex plane $\mathbb{C}$; for a suitable pair of sets $\bigl( \Omega, K \bigr)$,
the harmonic measure of $K$ with respect to $\Omega$, evaluated at a point $z$, will
be denoted by $h(z, \Omega, K)$. (For the basic properties of the harmonic measure in this
setting, see, e.g., \cite{Ahl, Ran}.)

\begin{proposition}\cite{Ran}\label{ran}
Let $\Omega$ be an open, connected set in $\mathbb{C}$ such that its boundary has nonzero Hausdorff dimension, 
and let $K$ be a Borel subset of the boundary. Suppose that $u$ is a subharmonic function on $\Omega$ satisfying
\begin{align*}
  u(z) \le M,   & \ \mbox{for} \  z \in \Omega\\
  \limsup_{z \to \zeta} u(z) \le m, & \  \mbox{for}  \ \zeta \in K.
\end{align*}
Then
\[
 u(z) \le m \, h(z, \Omega, K) + M \bigl(1 - h(z, \Omega, K)\bigr) \ \ \mbox{for} \ z \in \Omega.
\]
\end{proposition}

The last ingredient needed is an extremal property of the harmonic
measure in the unit disk $\mathbb{D}$.

\medskip

\begin{proposition}\cite{Sol}\label{sol}
Let $\lambda$ be in $(0, 1)$, $K$ a closed subset of $[-1,1]$ 
such that $m^1(K) = 2\lambda$,
and suppose that the origin is in $\mathbb{D} \setminus K$. Then
\[
 h(0,\mathbb{D},K) \ge h(0,\mathbb{D}, K_\lambda) =
 \frac{2}{\pi} \arcsin \frac{1-(1-\lambda)^2}{1+(1-\lambda)^2}
\]
where $K_\lambda = [-1, -1+\lambda] \cup [1-\lambda, 1]$.
\end{proposition}

\medskip

At this point, it is beneficial to set a specific set of constants:
let $M$ be the solution to the equation $\frac{1}{2}h^*+(1-h^*)M=1$ where
$h^*=\frac{2}{\pi}\arcsin\frac{1-(\frac{3}{4})^\frac{2}{3}}{1+(\frac{3}{4})^\frac{2}{3}}$,
and let $c(M)$ be as in Theorem \ref{cpx} (note that $1<M<\frac{3}{2}$).

\medskip

\begin{theorem}\label{yay1}

Let $\omega$ be in $C\bigl( [0,T^*), L^\infty\bigr)$ where $T^*$ is the first possible blow-up time,
and assume, in addition, that $\omega_0$ is in $L^2$ (then, $\omega$ is automatically in
$L^\infty\bigl( (0,T^-), L^2 \bigr)$ for any $0 < T^- < T^* $).

\medskip

Let $t$ be an escape time, and suppose that there exists a temporal point
$s=s(t)$ as in (\ref{s}) such that for any spatial point
$x_0$, there exists a scale $\rho$, $0<\rho\le \frac{1}{2 c(M) \|\omega(s)\|_\infty^\frac{1}{2}}$
and a direction $\mathbf{d}$ with the property that the super-level set $V^{j, \pm}_s$ delineated in (\ref{V})
is 1D $(\frac{3}{4})^\frac{1}{3}$-sparse around $x_0$ at scale $\rho$; here, the pair $(j, \pm)$ is chosen 
according to the following
selection criterion: $|\omega(x_0, s)| = \omega_j^\pm(x_0, s)$.

\medskip

Then $T^*$ is not a blow-up time.
\end{theorem}

\medskip

\begin{remark}
Notice that--as far as the temporal intermittency goes--it is enough that the condition on
local 1D sparseness holds at some $s(t)$ for a single escape time $t$ (recall that there
are continuum-many escape times; in particular, there are continuum-many escape times
in any interval of the form $(T^*-\epsilon, T^*)$).
\end{remark}

\medskip

\begin{proof}

(The argument to be presented is a refinement of the argument introduced in \cite{Gr13}.)

Let $t$ be an escape time and $s=s(t)$ satisfy the condition (\ref{s}). The idea
is to show that

\begin{equation}\label{contr}
 \|\omega(s)\|_\infty \le \|\omega(t)\|_\infty
\end{equation}

utilizing the sparseness at $s$ via the harmonic measure maximum
principle; this would contradict the
assumption that $t$ is an escape and in turn the assumption that $T^*$ is the first 
possible blow-up time.

\medskip

Fix a spatial point $x_0$ and let $\mathbf{d}$ be the direction of 1D sparseness around $x_0$
postulated in the theorem. Since the 3D NSE are translationally and rotationally
invariant, and the coordinate transformations would change neither the lower bound
on the analyticity radius nor the $L^\infty$-bound on the complexified solution
(nor the sparseness), without
loss of generality, assume that $x_0 = 0$ and $\mathbf{d} = (1, 0, 0)$. (The argument
is completely local, i.e., we are considering one spatial point $x_0$ at the time.)

\medskip

Let us focus on the complexification of the real coordinate $x_1$ (regarding the other two
variables as parameters), and view $\omega(s)$ 
as an analytic, $\mathbb{C}^3$-valued function on the symmetric strip around the $x_1$-axis 
of the width of at least

\medskip

\[
\frac{1}{c(M) \|\omega(s)\|_\infty^\frac{1}{2}}.
\]

Since $\omega$ is analytic, the real and the imaginary parts of the component functions $\omega_j$ 
are harmonic, and their positive and negative parts subharmonic in the strip; in particular, the assumption
on the scale of sparseness $\rho$ implies that they are subharmonic in the open disc

\medskip

\[
 D_\rho = \{z \in \mathbb{C}: \; |z| < \rho\}.
\]

Select a pair $(j, \pm)$ such that $|\omega(0, s)|$ = $\omega_j^\pm(0, s)$. By the premise,
the corresponding set $V_s^{j,\pm}$ is 1D $(\frac{3}{4})^\frac{1}{3}$-sparse around 0 at scale $\rho$
(in the direction $e_1$).

\medskip

Define a set $K$ by

\medskip

\[
 K = \overline{ V_s^{j,\pm} \cap (-\rho, \rho)}
\]
 
(the complement is taken within the interval $[-\rho, \rho]$),
and note that $m^1(K) \ge 2 \rho \bigl(\frac{3}{4}\bigr)^\frac{1}{2}$.

\medskip

There are two disjoint scenarios to consider. In the first one, 0 is in $K$. Then,

\medskip

\[
|\omega(0, s)| = \omega_j^\pm(0, s) \le \frac{1}{2M} \|\omega(s)\|_\infty \le \frac{1}{2} \|\omega(t)\|_\infty;
\]

i.e., the contribution of this particular $x_0$ is consistent with (\ref{contr}), and we can move on.
In the second scenario, we can estimate
the harmonic measure of $K$ with respect to $D_\rho$ at 0 utilizing the sparseness in conjunction
with Proposition \ref{sol}.

\medskip

Recall that the harmonic measure is invariant with respect to any conformal map, and in particular
with respect to the scaling map $\phi(z) = \frac{1}{\rho} \, z$; hence, Proposition \ref{sol} yields

\medskip

\[
 h(0, D_\rho, K) \ge \frac{2}{\pi}\arcsin\frac{1-(\frac{3}{4})^\frac{2}{3}}{1+(\frac{3}{4})^\frac{2}{3}} = h^*.
\]

The harmonic measure maximum principle (Proposition \ref{ran}) applied to the subharmonic function $\omega_j^\pm$
then implies the following bound:

\medskip

\[
 \omega_j^\pm(0, s) \le h^* \, \frac{1}{2} \|\omega(t)\|_\infty + (1-h^*) \, M \|\omega(t)\|_\infty
 = \biggl( \frac{1}{2}h^*+(1-h^*)M \biggr) \|\omega(t)\|_\infty = \|\omega(t)\|_\infty;
\]

the estimate on $\omega_j^\pm$ on the complement of $K$ (within $D_\rho$) comes from the
estimate on the complexified solution in Theorem \ref{cpx}, and the last equality from our
choice of the parameters. 

\medskip

Consequently, the second scenario is also compatible with the inequality (\ref{contr}).

\medskip

Since the point $x_0$ was an arbitrary point in $\mathbb{R}^3$, this concludes the proof.

\end{proof}

\begin{remark}\label{uloc}
Note that the argument utilizing the sparseness via the harmonic measure maximum principle \emph{per se}
is completely local. It is the application of Theorem \ref{cpx} that ascribes to Theorem \ref{yay1} the uniformly local
nature. To arrive to a completely local result, it suffices to localize the result presented in Theorem \ref{cpx}; c.f., 
Remark \ref{loc}.
\end{remark}

\bigskip

\section{Spatial complexity of the vorticity components:
an a priori bound}

\medskip

Concepts of a set being $r$-mixed (or `mixed to scale $r$')
and $r$-semi-mixed appear in the study of 
mixing properties of incompressible (or nearly incompressible)
flows, and in particular in the study of the optimal mixing strategies for passive scalars 
(e.g., density of a tracer) advected
by an incompressible velocity field (see, e.g., \cite{Br} and \cite{IKX}). 

\medskip

\begin{definition}\label{semi_mixed} 
Let $r>0$.  An open set $S$ is $r$-\emph{semi-mixed} if there exists $\delta$ in $(0,1)$ 
such that
$$
\frac{m^3(S\cap B(x,r))}{m^3(B(x,r))} \leq \delta
$$ 
for every $x$ in $\mathbb{R}^3$.
If the complement of $S$ is $r$-semi-mixed as well, then $S$ is said to be 
$r$-\emph{mixed}.
\end{definition}

\medskip

\begin{remark}
Recalling the definitions from the previous section, note that if the set $S$ is $r$-semi-mixed 
(with the ratio $\delta$), then it is 3D $\delta$-sparse around every point $x_0$ in $\mathbb{R}^3$ 
at scale $r$.
\end{remark}

\medskip

The following lemma is a vector-valued version of a positive scalar-valued
lemma in \cite{IKX} where the application was to the density of a tracer;
a vectorial Besov space version was given in \cite{FGL} in the context of the
$B^{-1}_{\infty, \infty}$-regularity criterion on the velocity. We present a 
detailed proof for completeness of the exposition.

\medskip

\begin{lemma}\label{mixing_lemma}
Let  $r \in (0,1]$ and $f$ a bounded, continuous vector-valued function on $\mathbb{R}^3$.
Then, for any pair $(\lambda, \delta)$, $\lambda$ in $(0,1)$ and $\delta$ in $\bigl(\frac{1}{1+\lambda}, 1\bigr)$, 
there exists a constant 
$c^*(\lambda, \delta) > 0$ such that if 

\[
\|f\|_{H^{-1}}  \leq  c^*(\lambda, \delta) \, r^\frac{5}{2} \|f\|_\infty
\]

then each of the six super-level sets $S_\lambda^{i, \pm} = \left\{ x\in \mathbb{R}^3: f_i^{\pm}(x) > 
\lambda \|f\|_\infty \right\}$ is 
$r$-semi-mixed with the ratio $\delta$. ($H^{-1}$ denotes the dual of the Sobolev space $H^1$.)
\end{lemma} 

\medskip

\begin{remark}

It is worth noting that there could be a significant discrepancy between semi-mixedness of the full vectorial
super-level sets and the super-level sets of the components. As a matter of fact, the former may not be
semi-mixed at all, while the latter may be $r$-semi-mixed at any predetermined scale.
In particular, a duality argument of the type utilized in the lemma does
not seem to lead to any quantifiable semi-mixedness of the full vectorial super-level sets.

\end{remark}

\medskip

\begin{proof}

Assume the opposite, i.e., there exists an index $i$ such that either 
$S^{i,+}_\lambda$ or $S^{i,-}_\lambda$ is not $r$-semi-mixed with the ratio $\delta$.
Suppose that it is $S^{i,+}_\lambda$ (if it were $S^{i,-}_\lambda$, the
only modification would be to replace the function $\phi$ below with $-\phi$).

\medskip

Consequently, 
there exists a spatial point $x_0$ such that

\[
m^3(S^{i,+}_\lambda \cap B(x_0,r)) > \delta \, V_3 \, r^3, 
\]

where $V_3$ denotes the volume of the unit ball in $\mathbb{R}^3$. 

\medskip

Let $\phi$ be an $H^1$-optimal, smooth radial (monotone) function, equal to 1 in $B(x_0,r)$, and vanishing
outside $B(x_0, (1+\eta)r)$ for some $\eta>0$ (the value to be determined). Then, by duality,

\begin{equation}\label{norm_u}
\|f\|_{H^{-1}} \geq \frac{1}{\|\phi\|_{H^1}} \abs {\int_{\mathbb{R}^3} f_i(x) \phi(x) \, dx}.
\end{equation}

\medskip

An explicit calculation of the $H^1$-norm of $\phi$ yields

\begin{equation}\label{norm_f}
\|\phi\|_{H^1}  \leq c(\eta)  \, r^\frac{1}{2}
\end{equation}

for a suitable $c(\eta)>0$ (recall that $r$ is in $(0, 1]$).

\medskip

The objective is to obtain a lower bound on the numerator by performing a suitable 
domain-decomposition of the integral.

\begin{align*}
\abs{\int_{\mathbb{R}^3} f_i(x )\phi(x)\, dx} \geq \int_{\mathbb{R}^3} f_i(x) \phi(x)\, dx
\geq I - \abs{II} - \abs{III}, 
\end{align*}

where 

\[
I = \int_{S^{i,+}_\lambda \cap B(x_0,r)} f_i(x) \phi(x)\, dx,
\]

\[
II = \int_{B(x_0,r)\backslash S^{i,+}_\lambda} f_i(x) \phi(x)\, dx
\]

and 

\[
III = \int_{\bigl(B(x_0,(1+\eta)r)\backslash B(x_0,r)\bigr)} f_i(x) \phi(x)\, dx.
\]

\medskip

It is transparent that

\begin{align}\label{I}  
I &= \int_{S^{i,+}_\lambda \cap B(x_0,r)} f_i(x) \, dx =  \int_{S^{i,+}_\lambda \cap B(x_0,r)} f_i^+ (x) \,dx \notag \\ 
&  > \lambda \, \|f\|_\infty \, m^3(S^{i,+}_\lambda \cap B(x_0, r)) \geq  \lambda \, \delta \, V_3 \, r^3 \|f\|_\infty,
\end{align}

\begin{align}\label{II}
\abs{II} & = \abs{\int_{B(x_0,r) \backslash S^{i,+}_\lambda} f_i(x)\ , dx} \leq \|f\|_\infty \left( m^3(B(x_0,r) - m^3(S^{i,+}_\lambda \cap B(x_0,r))\right)\notag \\
& \leq \|f\|_\infty \left(V_3 \, r^3 - \delta \, V_3 \, r^3\right) \notag \\
& = \left(1 - \delta \right) \, V_3 \, r^3 \|f\|_\infty
\end{align}

and 

\begin{align}\label{III}
\abs{III} & \leq \abs{\int_{(B(x_0,(1+\eta)r)\backslash B(x_0,r))} f_i(x) \, dx}\notag \\
& \leq \|f\|_\infty \left( m^3(B(x_0,(1+\eta)r) - m^3(B(x_0,r))\right)\notag \\
& \leq \left((1+\eta)^3 - 1\right) \, V_3 \, r^3 \, \|f\|_\infty. 
\end{align}

\medskip

Collecting the estimates \eqref{norm_u}, \eqref{norm_f} and \eqref{I}--\eqref{III}, it follows that

\begin{align*}
\|f\|_{H^{-1}} > c^*(\eta) \, r^\frac{5}{2} \, \|f\|_\infty \,
(\lambda \delta + \delta  - (1+\eta)^3). 
\end{align*}

\medskip

Since $\delta > \frac{1}{1+\lambda}$ is postulated, the equation
$(1+\eta)^3 = \frac{\delta (1+\lambda) +1}{2}$ has a unique solution
$\eta=\eta(\lambda, \delta)$; this choice of $\eta$ yields
 
\begin{align}
\|f\|_{H^{-1}} > c^*(\lambda, \delta) \, r^\frac{5}{2} \, \|f\|_\infty
\end{align}

with $c^*(\lambda, \delta) = c^*(\eta) \, \frac{\delta (1+\lambda) -1}{2}$ (which is positive since
$\delta > \frac{1}{1+\lambda}$). This produces a contradiction.

\end{proof}

\medskip

The next theorem is a simple consequence of the above lemma and a careful choice of
parameters throughout the paper; it is designated a theorem
because of its significance.

\medskip

\begin{theorem}\label{yay2}
Let $u$ be a Leray solution (a global-in-time weak solution satisfying the global energy inequality),
and assume that $\omega$ is in $C\bigl( (0, T^*), L^\infty \bigr)$ for some $T^* > 0$.
Then
for any $\tau$ in $(0, T^*)$ for which the scale $r^*$ defined below is less or equal to one, 
the super-level sets

\[
 V^{j, \pm}_\tau=\biggl\{x \in \mathbb{R}^3: \, \omega_j^\pm(x,\tau) > \frac{1}{2M}\|\omega(\tau)\|_\infty\biggr\} 
\]

are 3D $\frac{3}{4}$-sparse around any spatial point $x_0$ at scale 

\[
 r^* =  c( \|u_0\|_2) \, \frac{1}{\|\omega(\tau)\|_\infty^\frac{2}{5}}
\]

where $c( \|u_0\|_2)$ is a constant depending only on the energy at time 0
($M$ is the parameter set preceding the statement of Theorem \ref{yay1}; recall
that $1 < M < \frac{3}{2}$). 

\end{theorem}

\begin{proof}

Notice that our choice of parameters implies that 

\[
 \frac{3}{4} > \frac{1}{1 + \frac{1}{2M}},
\]

and Lemma \ref{mixing_lemma} is applicable.

\medskip

Since we are on the whole space, an efficient way to estimate $\|\omega(\tau)\|_{H^{-1}}$
is by switching to the Fourier space where one is required to estimate integrals of the form

\[
I_{i,j}  =  \int \frac{1}{1+|\xi|^2} \, |\widehat{\partial_i u_j} (\xi, \tau)|^2 \, d \xi.
\]

This is plain since

\[
 I_{i,j} \le  \int \frac{|\xi_i|^2}{1+|\xi|^2} \, |\widehat{u_j} (\xi, \tau)|^2 \, d \xi \le \|\widehat{u(\tau)}\|_2^2 
 = \|u(\tau)\|_2^2 \le \|u_0\|_2^2
\]

(by the energy inequality).

\medskip

Consequently, in order to satisfy all the assumptions in the lemma, it suffices to postulate

\[
 c \, \|u_0\|_2 \le c^*\biggl(\frac{1}{1+\frac{1}{2M}}, \frac{3}{4}\biggr) \, r^\frac{5}{2} \, \|\omega(\tau)\|_\infty,
\]

which forces the choice of the scale of sparseness $r^*$ in the theorem.

\end{proof}

\medskip

\begin{remark}
For our purposes, the restriction $r^* \le 1$ in the theorem is irrelevant since we are only interested in
temporal points $\tau$ leading to a possible blow-up time. It is also more of a feature than a bug since
the formation of small scales is a characteristic of the fully nonlinear regime.
\end{remark}

\medskip

\begin{remark}
Some effort has been made to assure that both the cut-off levels for the super-level sets and
the ratios of sparseness in Theorem \ref{yay1} and Theorem \ref{yay2} are identical,
$\lambda = \frac{1}{2M}$ and $\delta = \frac{3}{4}$.
Moreover, since the constants featured
in either of the two scales of sparseness
are either absolute (after our choice of parameters was made) constants or 
the constant depending only on the initial energy of the solution, the tolerance parameter
$c_0$ depends only on the initial energy. Shortly, the regularity condition
is given in the class $Z_\frac{1}{2}\bigl(\frac{1}{2M}, \frac{3}{4}; c_0(\|u_0\|_2)\bigr)$, and
the \emph{a priori} bound in the class
$Z_\frac{2}{5}\bigl(\frac{1}{2M}, \frac{3}{4}; c_0(\|u_0\|_2)\bigr)$.
\end{remark}

\bigskip

\section{Epilogue}

\medskip

The authors' goal was to present a mathematical framework--the scale of
classes $Z_\alpha$--in which one could break the archetypal 3D NS scaling barrier
in the context of a blow-up-type argument.
More precisely--in this ambiance--we showed that the \emph{a priori} bound
can be shifted by an algebraic factor, from $Z_\frac{1}{3}$,
which corresponds to the classical energy-level \emph{a priori} bounds, to $Z_\frac{2}{5}$;
recall that the
regularity class in this setting is $Z_\frac{1}{2}$, which corresponds to the
classical regularity criteria, and can be viewed as a statement that the
radius of spatial analyticity is a \emph{bona fide} diffusion scale.

\medskip

Parallel to the continued mathematical efforts, there is an effort directed at
gaining some insight into a `typical value' of $\alpha$ by performing
fully resolved computational simulations of turbulent flows, as well as harvesting data from
the Johns Hopkins Turbulence Databases (JHTDB), and trying to identify
the scaling properties of the scale of sparseness of the super-level sets of the
positive and negative parts of the vorticity components with respect to the
diffusion scale $\frac{\nu^\frac{1}{2}}{\|\omega\|_\infty^\frac{1}{2}}$. This
has been a joint project with the complexity group of J. Mathiesen at the Niels Bohr 
Institute (NBI) in Copenhagen, and in particular, with M. Mizstal, and the ScienceAtHome
group at Aarhus University led by J. Sherson; J. Rafner, of both the NBI and Aarhus 
is the project coordinator.

\medskip

The simulations performed and the databases utilized cover both decaying turbulence (e.g., initialized with the Kida
vortex or with the frequency-localized noise) and forced isotropic turbulence; in all simulations/databases,
the spatial domain is a periodic box.

\medskip

The data have been collected and analyzed in three different regimes: deeply in the inertial range,
deeply in the dissipation range, and in the regime of the most interest in the studies of possible singularity
formation in the 3D NSE, the regime of transition between the two. 

\medskip

The results will be detailed in a separate publication. Here, we would like to report that there are indications of
a statistically significant dependence between the two scales (the geometric scale of sparseness and 
the analytic diffusion scale) in all three regimes; however, the transitional regime has been the one signified by 
the strong evidence of a power law dependence with the range of exponents implying that the 
solution in view stabilizes in a 
class $Z_\alpha$, for some $\alpha > \frac{1}{2}$,
indicating no obstruction to furthering the rigorous $Z_\alpha$-theory presented here. Figure 2. below provides
an example of this phenomenon.

\begin{center}
\begin{figure}[htbp]
\includegraphics[scale=.86]{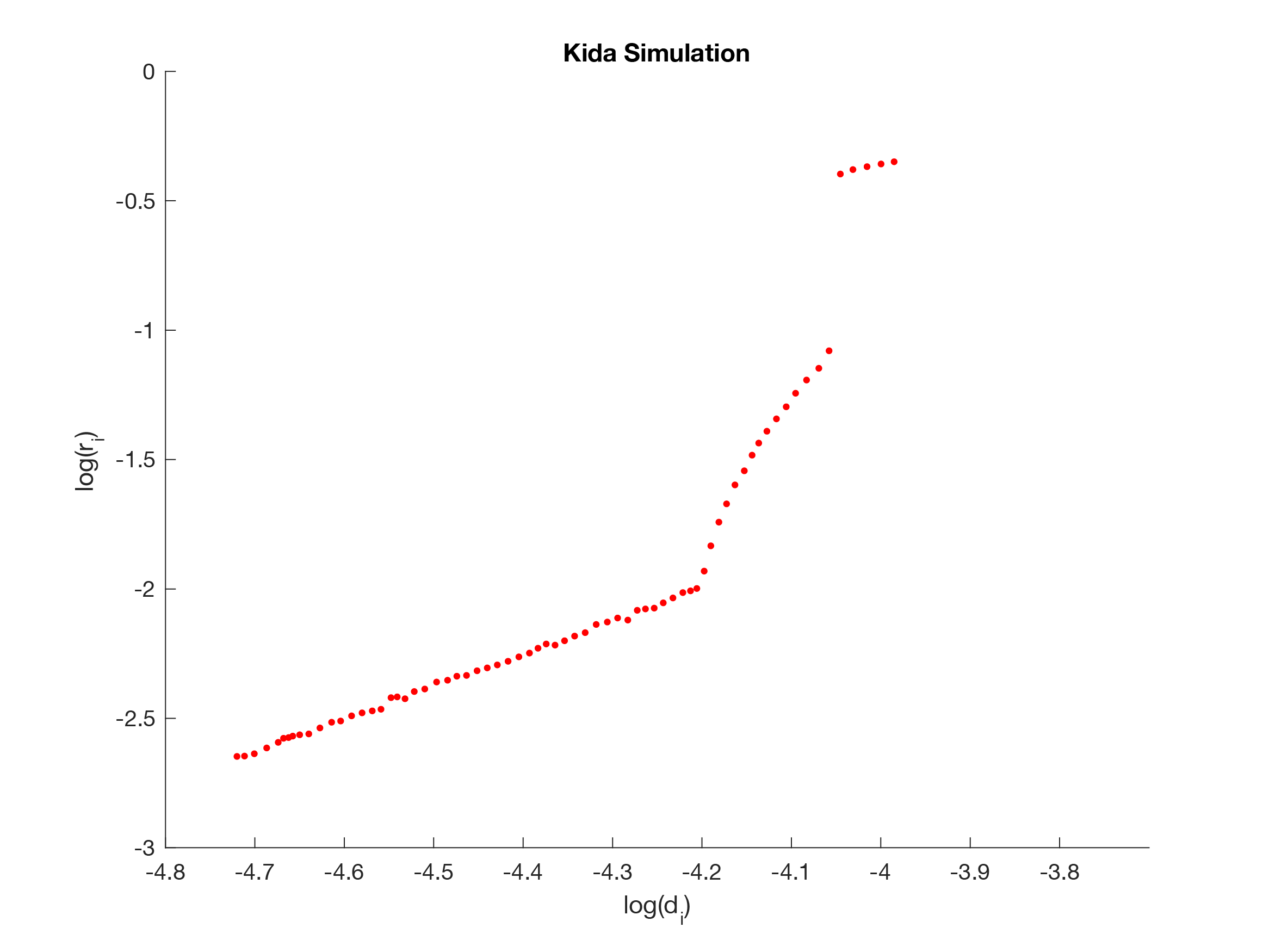}
\caption{\small{Log-log plot of the scale of 3D sparseness (of the vorticity components) $r_i=r_i(t_i)$ \emph{vs.} the diffusion scale
$d_i = \frac{\nu^\frac{1}{2}}{\|\omega(t_i)\|_\infty^\frac{1}{2}}$ from a simulation 
initialized with the Kida vortex (\cite{Ki}) initial condition at the Reynolds number of approximately $10^4$. Originating the flow at
the Kida vortex initial condition causes $\|\omega(t)\|_\infty$ to increase sharply before reaching the maximum 
and slumping, i.e., it produces a burst of $\|\omega(t)\|_\infty$ which is as close to `modeling a singularity' 
in the 3D NS flows as feasible. In the plot, the time runs from right to left; (approximately) the first third of 
the time-slices are sourced from the initial time-interval in which the flow still has `too much memory' of the initial condition, 
and the last two thirds of the time slices are sourced from the time-interval leading to
the peak. As illustrated in the plot--in the time-interval leading to the peak--a power-type relation of the form
$r \approx d^\alpha$ settles in; after proper rescaling, it is revealed that the slope of the line is approximately
equal to $\frac{6}{5}$, i.e., $r \approx d^\frac{6}{5}$. This implies that the solution in view 
(leading to the peak) is approximately in $Z_\frac{3}{5}$. Recall that
the energy-level bound is $Z_\frac{1}{3}$, our \emph{a priori} bound is $Z_\frac{2}{5}$, and what is needed for our
no blow-up criterion is at least $Z_\frac{1}{2}$. Hence, the plot indicates that further improvements of the
$Z_\alpha$ \emph{a priori} bounds--even past the critical class $Z_\frac{1}{2}$--might indeed be possible
(courtesy of M. Mizstal, NBI and J. Rafner, NBI and Aarhus).}}
 \end{figure}
 \end{center}

\newpage

\bigskip

\centerline{\textbf{Acknowledgments}}

\medskip

The work of A.F. is supported in part by the National Science Foundation grant DMS-1418911. The work
of Z.G. is supported in part by the National Science Foundation grant DMS - 1515805 and the
Lundbeck Foundation grant R217-2016-446 (a collaborative grant with the NBI and
Aarhus). Z.G. thanks the Program in Applied and Computational
Mathematics at Princeton for their hospitality in Spring 2017 when the paper was finalized. The authors
thank the referee for the constructive comments.

\end{document}